\documentclass[11pt, reqno]{article}
\usepackage{amsfonts}
\usepackage{amsmath}
\usepackage{amsthm}
\usepackage{amssymb}
\usepackage{lastpage}
\usepackage{fancyhdr}
\usepackage{fancybox}
\usepackage{mathrsfs}

\def\b{\mathbb }
\def\ch{{\rm ch}\>}
\def\sh{{\rm sh}\>}
\def\phi{\varphi }
\def\epsilon{\varepsilon}

\swapnumbers
\def\arch{{\rm arch}\>}

\theoremstyle{plain}
\newtheorem{theorem}{Theorem}[section]
\newtheorem{corollary}[theorem]{Corollary}
\newtheorem{lemma}[theorem]{Lemma}
\newtheorem{proposition}[theorem]{Proposition}
\theoremstyle{definition}

\newtheorem{remark}[theorem]{Remark}

\numberwithin{equation}{section}

\begin{document}

\title{ Central limit theorems for hyperbolic spaces and
  Jacobi processes on $[0,\infty[$}
\author{
Michael Voit\\
Fakult\"at Mathematik, Technische Universit\"at Dortmund\\
          Vogelpothsweg 87,
          D-44221 Dortmund, Germany\\
e-mail:  michael.voit@math.tu-dortmund.de
 }

\maketitle

\begin{abstract}
We present a unified approach to a
 couple of central limit theorems for radial random walks on hyperbolic spaces
and 
time-homogeneous Markov chains
on  $[0,\infty[$ whose transition probabilities are defined in terms of the Jacobi convolutions.
The proofs of all results are based on  limit results for the associated Jacobi functions.
 In particular, we consider 
 $\alpha\to\infty$,
 the case $\phi_{i\rho-\lambda}^{(\alpha,\beta)}(t)$
for small $\lambda$, and   $\phi_{i\rho-n\lambda}^{(\alpha,\beta)}(t/n)$
for $n\to\infty$.
The proofs of all these limit results are based on the
 known Laplace  integral representation for  Jacobi functions.
Parts of the results  are known, other improve
 known ones, and other are new.
\end{abstract}



\section{Introduction}

We here derive a couple of  central limit theorems for the distance of radial random walks 
$(S_n^k)_{n\ge0}$  from their starting points on the hyperbolic spaces $H_k(\b F)$ of
dimension $k\ge2$ over the fields  $\b F=\b R, \b C$, or the quaternions $\b H$.
The main observation for  proofs is that
 the distance processes are again Markov chains on $[0,\infty[$ whose transition probabilities
 are  related to to the product formula for the spherical functions for $H_k(\b F)$, i.e.,
 certain Jacobi functions. As
 all proofs work  without additional effort for more general ``Jacobi random walks''  on $[0,\infty[$,
i.e.,   Markov processes on $[0,\infty[$ whose transition probabilities
 are  related to  general Jacobi functions, 
we shall derive all results in this context.

To describe the main results, we regard  $H_k(\b F)$
 as symmetric space $G/K$ with
\begin{align}
\b F=\b R:& \quad\quad\quad G=SO_o(1,k), \quad K=SO(k)    \notag\\
\b F=\b C:& \quad\quad\quad G=SU(1,k), \quad K=S(U(1)\times U(k))    \notag\\
\b F=\b H:& \quad\quad\quad G=Sp(1,k), \quad K=Sp(1)\times Sp(k)    \notag
\end{align}
and define the dimension parameter 
$d:=dim_\b R\b F=1,2,4$.
We identify the double coset space  $G//K$ with the interval
$[0,\infty[$ such that $t\in[0,\infty[$ corresponds with the double coset
$$Ka_tK \quad\quad {\rm with} \quad\quad
a_t=\begin{pmatrix}\ch t& 0& \ldots&0&\sh t\\
              0&&&&0\\
             \vdots&&I_{k-1}&&\vdots\\
                   0&&&&0\\
                 \sh t& 0& \ldots&0&\ch t\end{pmatrix};
$$
see e.g. \cite{F} or \cite{Hel}. Using this homeomorphism $\phi:G//K\to [0,\infty[$,  we  define the
hyperbolic distance on $G/K$  by $d(xK,yK)=\phi(Ky^{-1}xK)$. 
In this way, $G$ acts on  $G/K=H_k(\b F)$ isometrically in a two-point homogeneous way, i.e., 
for  $x_1,x_2,y_1,y_2\in H_k(\b F) $ with $d(x_1,x_2)=d(y_1,y_2)$ there exists  $g\in G$ with
 $g(x_1)=y_1$ and $g(x_2)=y_2$. 

 Now consider a time-homogeneous Markov chain $(S_n^k)_{n\ge0}$ on  $H_k(\b F) $ 
with transition kernel $K_k$ starting at time 0 in  $eK\in G/K= H_k(\b F)$.
This Markov chain is called a radial random walk on  $H_k(\b F) $, if 
$K_k$  is $G$-invariant, i.e., if for all $g\in G$,  $x\in H_k(\b F) $, and  Borel sets
$A\subset  H_k(\b F) $, $K_k(g(x), g(A))=K_k(g,A)$. It is well-known (see e.g. Lemma 4.4 of \cite{RV1})
that  then for the canonical projection
$\pi:G/H=H_k(\b F)\to G//H=[0,\infty[$,  the image process
$(\pi(S_n^k))_{n\ge0}$ is  a  time-homogeneous Markov chain on $[0,\infty[$ with kernel 
$$\tilde K_k(x,A)=(\nu *\delta_x)(A)
 \quad\quad\text{for}\quad x\in [0,\infty[\quad\text{and Borel sets}\quad
A\subset [0,\infty[,$$
 where $\nu\in M^1([0,\infty[)$ is the distribution  of $d(S_n^k, S_{n+1}^k)$ (which is independent of 
$n$), $\delta_x$ is a point measure, and $*$ denotes the
the double coset convolution on  $G//H=[0,\infty[$.
We point out that the Markov kernel  $K_k$  and thus the finite-dimensional distributions of 
 $(S_n^k)_{n\ge0}$  are determined 
uniquely by  $\nu\in M^1([0,\infty[)$, and that the distributions  of the distances $d(S_n^k, eK)$
from the starting point $eK$ are 
 the $n$-th convolution powers $\nu^{(n)}$ of
 $\nu$ w.r.t.~$*$.

We are now interested in central limit theorems (CLTs) for $d(S_n^k, eK)$ for $n\to\infty$.

For the first type of results,
 we fix  $\nu\in M^1([0,\infty[)$ and $H_k(\b F)$,
 and introduce for each $c\in]0,1]$
the compressing map $D_{c}:x\mapsto c x$ on $[0,\infty[$ as well as the  compressed measure 
$\nu_{ c}:= D_{c}(\nu)   \in M^1([0,\infty[)$.
Now  consider the radial random walk 
$(S_n^{(k,c)})_{n\ge0}$ on $H_k(\b F)$ associated with the compressed  measure   $\nu_{ c}$.
We now look for  CLTs for $(S_n^{(k,  n^{-r})})_{n\ge0}$
depending on  $r\ge0$ which include or supplement the results  in  \cite{KTS}, \cite{Tu}, \cite{Gr},
 Section 3.2 of  \cite{Te},
  \cite{Tr},  and in \cite{Z1}, \cite{Z2}, and Section 7.4 of  the  monograph \cite{BH} in our setting.

The most classical case appears for $r=1/2$ and  is well-known; see
Theorem 7.4.1  of  \cite{BH} and references cited there.
Here, the distance processes
$\frac{1}{\sqrt n} d(S_n^{(k,  n^{-r})}, eK)$ from the starting point 
tend in distribution to some known limit depending 
on  $H_k(\b F)$ and the second moment of
 $\nu$. These limit distributions are known as (radial parts) of  Gaussians on   $H_k(\b F)$.
   As this case is studied  precisely in the literature, we omit this case here.

The case   $r=0$ without  initial compression is  due to \cite{Z1};
 see  Section 7.4 of  \cite{BH}.
We here improve this CLT by  a rate of convergence. 
The CLT needs the  modified moment functions
$$m_j(t):=m_j(d,k;t):=\int_0^1\int_0^\pi \left(\ln\bigl(|\ch t + r\cdot e^{i\phi}\sh t|\bigr)\right)^j
\> dm_{dk/2-1,d/2-1}(r,\phi) \quad\quad(j\in\b N)$$
with the  probability measure $m_{dk/2-1,d/2-1}\in M^1([0,1]\times[0,\pi])$ defined 
below in (\ref{mesasure-m}). It will  turn out in (\ref{moment-function})
below that this definition agrees with that in \cite{Z1} and   \cite{BH}.
 It is known that 
$m_1\ge0$ and $m_1(x)^2\le m_2(x)\le x^2$ for $x\ge 0$ with equality precisely for $x=0$.

\begin{theorem}\label{central-limit-momentenfkt-hyperbolisch}
Let $\nu\in M^1([0,\infty[ )$  with $\nu\ne
    \delta_0$ and
 finite second  moment.
  For $j\in\b N$ let $M_j:=\int_0^\infty m_j(t)\> d\nu(t)$
be the  modified moments of $\nu$ with $M_1,M_2<\infty$ by our assumption.
Then  $$\frac{d(S_n^{(k,1)},eK)-nM_1}{\sqrt n}$$
tends in distribution to $N(0, M_2-M_1^2)$ with $ M_2-M_1^2>0$.

Moreover, if  $\nu\in M^1([0,\infty[ )$ in addition has a finite third moment, then
 the distribution functions of the random variables above 
tend uniformly on $\b R$ to the distribution function of
$N(0, M_2-M_1^2)$ of order $n^{-1/3}$.
\end{theorem}

The case $r>1/2$ was studied in \cite{V2}.
 We  reprove the result here in a different way:

\begin{theorem}\label{central-flacher-limes-hyperbolisch}
Let $\nu\in M^1([0,\infty[ )$  with $\nu\ne
    \delta_0$ and
 finite second moment
$m_2:= \int_0^\infty x^2\> d\mu(x)$
(which is automatically positive).
Then
  $$ \left(\frac{dk}{m_2}\right)^{1/2}\cdot  n^{r-1/2} \cdot d(S_{n}^{(k, n^{-r})},eK) $$ 
 tends 
 in distribution  to the Rayleigh distribution $\rho_{dk/2-1}$,
 where $\rho_\alpha$  
has the Lebesgue
 density
$$\frac{1}{2^\alpha \Gamma(\alpha+1)} x^{2\alpha+1} e^{-x^2/2} \quad\quad(x\ge 0).$$
\end{theorem}

Notice that in the preceding result  $dim_{\b R}H_k(\b F)=dk$  and that
  $\rho_{dk/2-1}$ is precisely the radial part 
of a  $dk$-dimensional standard normal distribution on $\b R^{dk}$.
The preceding theorem therefore means that for $r>1/2$, the CLT
 forgets the curvature  and admits the  behavior of classical sums of 
i.i.d. random variables on the tangent space.

The case $r\in ]0,1/2[$ was not considered before. We here obtain:

\begin{theorem}\label{central-neuer-fall-hyperbolisch}
Let $\nu\in M^1([0,\infty[ )$  with $\nu\ne
    \delta_0$, with  compact support, and thus with finite
 moments
$m_l:= \int_0^\infty x^l\> d\mu(x)\in ]0,\infty[$ ($l\in\b N$).
 Then:
\begin{enumerate}
\item[\rm{(1)}] For $r\in]1/6, 1/2[$, 
 $$\frac{ d(S_{n}^{(k, n^{-r})},eK) -\frac{(d(k+1)/2-1) m_2 n^{1-2r}}{dk}}{ n^{1/2-r}}$$
tends in distribution to $N(0, \frac{m_2}{dk})$.
\item[\rm{(2)}]  If $r=1/6$, then the random variables of part (1)
tend in distribution to
 $N(-M, \frac{m_2}{dk})$
with the mean
$$M:=-\frac{(d(k+1)/2-1)d(k+3)/2-2)\cdot m_4}{6dk(dk/2 +1)}.$$
\item[\rm{(3)}] If $r\in ]0,1/6[$, then
$$\frac{ d(S_{n}^{k, n^{-r})},eK) -\frac{(d(k+1)/2-1) m_2}{dk}\cdot n^{1-2r}}{ n^{1-4r}} \longrightarrow
  M   $$
in probability with $M$ as in (2).
\end{enumerate}
\end{theorem}

Besides the preceding limit theorems for a fixed hyperbolic space, we  also derive the following 
CLT for a fixed field $\b F$, where the dimension $k$ and the number $n$ of steps  tend to infinity.
It generalizes a result in \cite{V3}:

\begin{theorem}\label{hyperbolic-clt-special}
Let  $(k_n)_{n\ge1}\subset\b N$ be increasing  with 
 $\lim_{n\to\infty}n/ k_n=0$, and fix $\b F$ as above. 
Let $\nu\in M^1([0,\infty[)$ with  finite second moment $\int_0^\infty x^2 \>
d\nu(x)$, and consider the associated radial random walks $(S_n^k)_{n\ge0}$ on $H_k(\b F)$ for $k\in\b N$.
Then,  $m_j:= \int_0^\infty (\ln(\ch x))ĵ\>
d\nu(x)<\infty$ exist for $j=1,2$, and
$$\frac{d(S_n^{k_n},S_0^{k_n})  -nm_1}{\sqrt n}$$
tends in distribution for $n\to\infty$ to 
$N(0,m_2-m_1^2)$. 
\end{theorem}

An extension of this CLT without the restriction  $\lim_{n\to\infty}n/ k_n=0$ was
recently derived by Grundmann \cite{G} by using completely different methods.

We now briefly describe the common roots of the proof of the preceding CLTs.
We regard
the spherical functions of the Gelfand pair $(G,H)$  above as continuous functions on $[0,\infty[$ 
which are multiplicative w.r.t. $*$, i.e., $f(x)f(y)=\int_0^\infty f\> d(\delta_x*\delta_y)$
for $x,y\ge0$. It is well-known (see  \cite{Ko2}) that in our case  all spherical functions are given by
 Jacobi functions
\begin{equation}
\phi_\lambda^{(\alpha,\beta)}(t):= \>_2F_1((\alpha+\beta+1-i\lambda)/2, (\alpha+\beta+1+i\lambda)/2;
\alpha+1; -\sh^2 t) \quad\quad (\lambda\in\b C)
\end{equation}
with the parameters
\begin{equation}\label{parameter-hyp}
\alpha= dk/2-1, \quad\quad\quad \beta=d/2-1 \quad\quad\text{with}\quad
d:=dim_{\b R}\b F=1,2,4.
\end{equation}
Moreover, the double coset convolutions $*$ on $[0,\infty[$ for the hyperbolic spaces above 
 can be regarded
as special cases of   Jacobi convolution $*_{(\alpha,\beta)}$ on $[0,\infty[$
  which were investigated mainly
by Flensted-Jensen and Koornwinder. In the following we refer to the survey  \cite{Ko2}
on the subject.
 For $\alpha>\beta\ge -1/2$ with $\alpha>-1/2$, this convolution is given by
\begin{equation}
\delta_s *_{(\alpha,\beta)}\delta_t (f) :=
 \int_0^1\int_0^\pi  f\bigl(\arch\bigl| \ch s\cdot\ch t +re^{i\phi}\sh s \cdot\sh t\bigr|\bigr)
dm_{\alpha,\beta} (r,\phi)
\end{equation}
for $f\in C_b([0,\infty[)$ and for the probability measure  $dm_{\alpha,\beta}$
with
\begin{equation}\label{mesasure-m}
dm_{\alpha,\beta}(r,\phi)= \frac{2\Gamma(\alpha+1)}{\Gamma(1/2)\Gamma(\alpha-\beta)\Gamma(\beta+1/2)}
\cdot (1-r^2)^{\alpha-\beta-1}(r\sin\phi)^{2\beta}\cdot r \> dr\> d\phi
\end{equation}
for $\alpha>\beta>-1/2$. For  $\alpha>\beta=-1/2$, the measure degenerates into
\begin{equation}\label{mesasure-md1}
dm_{\alpha,-1/2}(r,\phi)=\frac{2\Gamma(\alpha+1)}{\Gamma(1/2)\Gamma(\alpha+1/2)} 
 (1-r^2)^{\alpha-1/2}  dr\cdot
\frac{1}{2}d(\delta_0+\delta_\pi)(\phi),
\end{equation}
and for  $\alpha=\beta>-1/2$ into
\begin{equation}\label{mesasure-md2}
dm_{\alpha,\alpha}(r,\phi)=\frac{2\Gamma(\alpha+1)}{\Gamma(1/2)\Gamma(\alpha+1/2)} \sin^{2\alpha}\phi
\> d\phi \cdot d\delta_0(r).
\end{equation}

Now fix $\alpha\ge\beta\ge -1/2$ with $\alpha>-1/2$. It is well-known that the Jacobi
convolution above can be extended 
uniquely in a weakly continuous, bilinear way to a probability-preserving
convolution $*_{(\alpha,\beta)}$ on  $M_b([0,\infty[)$, and that one
obtains the  so-called Jacobi-type hypergroups on $[0,\infty[$; see \cite{Ko2}, \cite{BH}, \cite{Tr}.

Using this  Jacobi-convolution, we  now generalize  the Markov processes 
$(d(S_n^k,eK))_{n\ge0}$  above as follows: 
Fix a
 measure $\nu\in M^1([0,\infty[)$, and consider a time-homogeneous Jacobi
random walk $(S_n^{(\alpha,\beta)})_{n\ge0}$ on $[0,\infty[$ with law $\nu$ of index $(\alpha,\beta)$, i.e., a
time-homogeneous Markov process on $[0,\infty[$ starting at  $0$ with transition
probability
$$P(S_{n+1}^{(\alpha,\beta)}\in A|\> S_n^{(\alpha,\beta)}=x)= (\delta_x *_{(\alpha,\beta)}\nu)(A)
\quad\quad(x\ge0, \> A\subset [0,\infty[ \quad\text{a Borel set}).$$
This notion agrees with that for $(d(S_n^k,eK))_{n\ge0}$ in the hyperbolic case above.
We shall derive all CLTs above  in this more general setting.
 We shall do this 
in Section 3 for growing dimensions and in Section 4 for fixed dimension $k$. In this way, Theorems
\ref{central-flacher-limes-hyperbolisch},
\ref{central-limit-momentenfkt-hyperbolisch}, \ref{central-neuer-fall-hyperbolisch}, and
\ref{hyperbolic-clt-special} are just special cases of Theorems
\ref{central-flacher-limes}, \ref{central-limit-momentenfkt}, \ref{central-neuer-fall}
and \ref{central-limit-alpha-infty}
below respectively. The proofs of all these limit theorems will be be based on several limit results for
 Jacobi functions which we will derive in Section 2. The basis of all these
 limit results will be the following  well-known Laplace integral representation for the Jacobi
functions; see Section 5.2 of \cite{Ko2}:

\begin{theorem}\label{integral-representation}
Let $\alpha\ge\beta\ge -1/2$. Then, for $\lambda\in\b C$ and $t\ge0$,
$$\phi_{\lambda}^{(\alpha,\beta)}(t)=\int_0^1\int_0^\pi |\ch t +re^{i\phi}\sh t|^{i\lambda-\rho}
 \> dm_{\alpha,\beta}(r,\phi)$$
with 
$$\rho:=\alpha+\beta+1\ge0$$
and   the probability measure  $dm_{\alpha,\beta}$ introduced in 
(\ref{mesasure-m}),(\ref{mesasure-md1}) and (\ref{mesasure-md2}) respectively.
\end{theorem}

We mention that Theorem \ref{integral-representation} admits an   analogue for Jacobi 
polynomials   due to Koornwinder \cite{Ko1},
 and that this integral representation  also leads
to  limit results. In particular, one can revisit
Hilb's formula for Jacobi polynomials (see \cite{Sz}) in order to prove
CLTs for Markov chains on $\b Z_+$  associated  with
 orthogonal polynomials; see \cite{Ga}, \cite{V1}, Section
7.4 of  \cite{BH}, and references cited there for the topic.

We also mention that  parts of this paper can be extended 
 to certain families of Heckman-Opdam hypergeometric functions
 of type BC which  include the spherical functions for the symmetric spaces
$SU(p,q)/(SU(p)\times SU(q))$. For the 
background on these functions  and the associated
convolution structures on Weyl chambers of type B we refer to
 \cite{H},\cite{HS}, 
\cite{O} and \cite{R}. In \cite{RV2} we  generalize the Harish Chandra
 integral representation for  $SO_0(p,q)/(SO(p)\times SO(q))$ for
 $\b F=\b R$ of \cite{Sa} to the more general setting considered
 in \cite{R}; this integral representation is similar to that in Theorem \ref{integral-representation}
and  leads with more technical effort 
to multidimensional  extensions of some of the result in the present paper.

\section{Limit relations for Jacobi functions}

 We start with two major
 results where $\alpha$ or both parameters $\alpha,\beta$ converge to infinity.

\begin{proposition}\label{limit-alpha-infty}
Let $\beta\ge-1/2$. Then there exists a constant $C=C(\beta)$ such that
 for all $t\ge0$, $\alpha>max(\beta,0)$, and 
$\lambda\in\b R$ 
$$\Bigl|\phi_{i\rho-\lambda}^{(\alpha,\beta)}(t)- e^{i\lambda\cdot\ln(\ch t)}\Bigr|\le
 C\frac{|\lambda|\cdot min(1, t)}{\alpha^{1/2}}.$$
\end{proposition}

\begin{proof}
Using Theorem \ref{integral-representation}, we consider the difference 
$$R:=\Bigl|\phi_{i\rho-\lambda}^{(\alpha,\beta)}(t)- e^{i\lambda\cdot\ln(\ch t)}\Bigr|
=\Bigl|\int_0^1\int_0^\pi \left(|\ch t +re^{i\phi}\sh t|^{i\lambda}
- |\ch t|^{i\lambda}\right) \> dm_{\alpha,\beta}(r,\phi)\Bigr|,$$
which satisfies
$$R\le \int_0^1\int_0^\pi |g(re^{i\phi},t)-1|\> dm_{\alpha,\beta}(r,\phi)$$
with 
$$g(re^{i\phi},t):=|1+re^{i\phi}\cdot \sh t/\ch t|^{i\lambda}=
e^{i\lambda\cdot\ln(|1+re^{i\phi}\cdot \sh t/\ch t|)}.$$
As $|e^{ix}-1|\le\sqrt 2\cdot |x|$ for $x\in\b R$, we have
$$ |g(re^{i\phi},t)-1|\le \sqrt 2\cdot|\lambda|\cdot |\ln(|1+re^{i\phi}\cdot \sh t/\ch t|)|.$$
Moreover, as
for $z\in \b C$ with $|z|<1$
$$|\ln(|1+z|)|\le |\ln(1+z)|\le |z|+ |z|^2+|z|^3\ldots =|z|/(1-|z|),$$
and as
$0\le   \sh t/\ch t\le min(1,t)$ for $t\ge0$, we obtain for $0\le r\le 1$ and
$t\ge0$ 
 $$ |g(re^{i\phi},t)-1|\le \sqrt 2\cdot|\lambda|\cdot 
\frac{r\cdot \sh t/\ch t}{1-r\cdot \sh t/\ch t}\le
 \sqrt 2\cdot|\lambda|\cdot  min(1,t)\frac{r}{1-r}.$$
Therefore, as $r\in[0,1]$,
 $$ |g(re^{i\phi},t)-1|\le 2\sqrt 2\cdot|\lambda|\cdot  min(1,t)\frac{r}{1-r^2}.$$
Now consider the probability measure
 $m_{\alpha,\beta}$ for $\alpha>\beta>-1/2$ introduced in Theorem 
\ref{integral-representation}. We conclude that
$$R\le 2\sqrt 2\cdot|\lambda|\cdot 
\frac{2\cdot  min(1,t)\cdot  \Gamma(\alpha+1)}{\Gamma(1/2)\Gamma(\alpha-\beta)\Gamma(\beta+1/2)}
\int_0^\pi \sin^{2\beta}\phi\> d\phi \cdot 
\int_0^1 (1-r^2)^{\alpha-\beta-2} r^{2\beta+2}\> dr.$$  
 Using standard formulas for the  beta-integrals on the right-hand side and
 finally
 $$\Gamma(\alpha+1)/\Gamma(\alpha+1/2)=O(\sqrt\alpha)\quad\quad (\alpha\to\infty),$$ we obtain
$$R\le  8 \cdot|\lambda|\cdot  min(1,t)
 \frac{\Gamma(\beta+3/2)}{\Gamma(\beta+1)}\cdot
\frac{\Gamma(\alpha+1)}{(\alpha-\beta-1)\Gamma(\alpha+1/2)} \quad=\quad
|\lambda|\cdot  min(1,t) \cdot
O(1/\sqrt\alpha)$$
as claimed. The case $\beta=-1/2$  follows in the same way 
from the definition of
$m_{\alpha,\beta}$ .
\end{proof}

We now consider a variant  where
$\alpha$ and $\beta$ tend to infinity in a coupled way.
 We  note that the convergence results \ref{limit-alpha-infty}
 and \ref{limit-alpha-beta-infty}
(without error estimates) correspond  to  well-known limit transitions  for Jacobi polynomials 
which can be found e.g. in  \cite{Ko3} or \cite{Sz}.

\begin{proposition}\label{limit-alpha-beta-infty}
Fix constants $c>1$ and $d>0$ and put $\alpha:=c\beta+d$.
 Then there exists a constant $C=C(c,d)$ such that
 for all $t\ge0$, $\beta>0$, and 
$\lambda\in\b R$ 
$$\Bigl|\phi_{i\rho-\lambda}^{(c\beta+d,\beta)}(t)-
 e^{i\lambda\cdot\ln\sqrt{\ch^2 t- (1/c)\sh^2t}}\Bigr|\le
 C\frac{|\lambda|\cdot min(1, t)}{\beta^{1/2}}.$$
\end{proposition}

The proof will be based on the following observation which is likely well-known.

\begin{lemma}\label{asymptotic-power} 
Consider a continuous function $f:[0,1]\to[0,\infty[$ such that there
exist $x_0\in[0,1]$ and constants $0<c_1\le c_2\le1$ such that
 $f(x)\le 1-c_2(x-x_0)^2$ holds for all $x\in[0,1]$, and 
 $f(x)\ge 1-c_1(x-x_0)^2$ for all $x\in[0,1]$ in a suitable neighborhood of
$x_0$. Moreover, let $g:[0,1]\to[0,\infty[$ be continuous with $g(x_0)>0$.
Then, for continuous $n\to\infty$,
$$\int_0^1 |x-x_0|\cdot f^n(x)\cdot g(x) \> dx =
 O\left(\frac{1}{\sqrt n}\int_0^1  f^n(x)\cdot g(x)\> dx\right).$$ 
\end{lemma}

\begin{proof} For $n$ sufficiently large, we obtain by continuity arguments and omitting parts
of the left hand integral that
$$\int_0^1  f^n(x)\cdot g(x)\> dx\ge \int_0^{1/\sqrt{c_1n}} (1-c_1x^2)^ng(x_0)/2\> dx=
\frac{g(x_0)}{2\sqrt{c_1 n}}\int_0^1  (1-y^2/n)^n \> dy$$
where the integral on the right hand side converges to some positive
constant. On the other hand,
\begin{align}
\int_0^1 |x-x_0|\cdot f^n(x)\cdot g(x)\> dx &=\left(\int_0^{x_0}+\int_{x_0}^1\right)
|x-x_0|\cdot f^n(x)\cdot g(x)\> dx\notag\\
&\le 2\int_0^1 x(1-c_2x^2)^n\cdot h(x)\> dx
\notag\\
 &=\frac{2}{n\sqrt c_2}\int_0^{\sqrt{c_2n}}
(1-y^2/n)^n \cdot h(y/\sqrt{n c_2})\> dy
\notag
\end{align}
for some continuous function $h$ depending on $g$
where the integral on the right hand side converges to some finite positive
constant. A combination of both results leads to the lemma.
\end{proof}

\begin{proof}[Proof of Proposition \ref{limit-alpha-beta-infty}]
Precisely as in the proof of Proposition \ref{limit-alpha-infty} we obtain
that
$$R:=\Bigl|\phi_{i\rho-\lambda}^{(\alpha,\beta)}(t)-e^{i\lambda\cdot\ln\sqrt{\ch^2 t- (1/c)\sh^2t}} \Bigr|
\le 2|\lambda|\int_0^1\int_0^\pi |\ln A|\> dm_{\alpha,\beta}(r,\phi)$$
for
\begin{equation}\label{def-A}
A:=\frac{ |\ch t +re^{i\phi}\sh t|}{|\ch t +i\cdot\frac{1}{\sqrt c}\sh t|}.
\end{equation}
As for $r\in[0,1]$ and $c>1$, we have
$(1-r)/2\le A\le 1+r$, we obtain
$$|A-1|\le max\left( r, (r+1)/2\right)=(r+1)/2.$$
Therefore, using $|\ln(|1+z|)|\le|z|/(1-|z|)$ for $|z|=|A-1|\le1$ as in the
preceding proof, we conclude that
$$ |\ln A|\le \frac{|A-1|}{1-|A-1|}\le\frac{2}{1-r}\cdot \frac{|1-A^2|}{|1+A|}
\le \frac{2}{1-r}\cdot |1-A^2|\le \frac{4}{1-r^2}\cdot |1-A^2|.$$
Moreover, defining $\tau:=\sh t/\ch t\le min(t,1)$, we have
\begin{align}
|1-A^2|&= \left|1-\frac{ (1+r\tau\cos\phi)^2
  +r^2\tau^2\sin^2\phi}{1+\tau^2/c}\right|
= \left|\frac{\tau^2(1/c-r^2) +2r\tau\cos\phi}{1+\tau^2/c}\right|
\notag\\
&\le \tau\left(2 |\cos\phi|+ \left|\frac{|1/c-r^2|}{1+\tau^2/c}\right|\right)
\le  2\cdot\tau\left( |\cos\phi|+ |r-1/\sqrt c|\right).
\notag
\end{align}
In summary, we have
\begin{align}
R&\le 16|\lambda|\tau\cdot \int_0^1\int_0^\pi \frac{ |\cos\phi|+ |r-1/\sqrt
  c|}{1-r^2}\> dm_{\alpha,\beta}(r,\phi)
\notag\\
&= 16|\lambda|\tau\cdot
\frac{2\Gamma(\alpha+1)}{\Gamma(1/2)\Gamma(\alpha-\beta)\Gamma(\beta+1/2)}
\notag\\
&\cdot
\int_0^1\int_0^\pi 
(|\cos\phi|+ |r-1/\sqrt
  c|)\cdot    ((1-r^2)r^{2/(c-1)})^{\beta(c-1)-2}\cdot
  r^{1+4/(c-1)}(1-r^2)^{d}\cdot\sin^{2\beta}\phi\> dr\> d\phi.
\notag
\end{align}
We now apply Lemma \ref{asymptotic-power} to $g(r):= r^{1+4/(c-1)}(1-r^2)^{d}$
and $f(r):=(1-r^2)r^{2/(c-1)}$ with the maximum value of $f$ on $[0,1]$  at
$r_0=1/\sqrt c$ and notice that  $dm_{\alpha,\beta}(r,\phi)$ is a probability measure. This yields 
$$\int_0^1 |r-1/\sqrt
  c|)\cdot  ((1-r^2)r^{2/(c-1)})^{\beta(c-1)-2}\cdot
  r^{1+4/(c-1)}(1-r^2)^{d}\> dr =O(1/\sqrt\beta).$$
As a similar argument also yields
$$\int_0^\pi |\cos\phi|\cdot\sin^{2\beta}\phi\> d\phi= O(1/\sqrt\beta),$$
we obtain $R\le 16|\lambda|\tau\cdot O(1/\sqrt\beta)$ as claimed.
\end{proof}

We next turn to a limit concerning Bessel functions. 
 Recapitulate that the normalized Bessel functions 
$$j_\alpha(t):=_0F_1(\alpha+1;-t^2/4)=\Gamma(\alpha+1) \cdot \sum_{n=0}^\infty \frac{(-1)^n
 (t/2)^{2n}}{n!\> \Gamma(n+\alpha+1)}$$
with $j_\alpha(0)=1$ for $\alpha>-1/2$ admit the integral representation
\begin{equation}\label{int-rep-bessel}
j_\alpha(t)=\frac{\Gamma(\alpha+1)}{\Gamma(1/2)\Gamma(\alpha+1/2)}\int_{-1}^1 e^{itu}(1-u^2)^{\alpha-1/2}\> du.
\end{equation}
In order to compare $j_\alpha$ with the Jacobi functions, we rewrite it  as
\begin{equation}\label{komp-int-rep-bessel}
j_\alpha(t)=\int_0^1\int_{0}^\pi e^{itr \cos \phi}\> dm_{\alpha,\beta}(r,\phi)
\end{equation}
for $\alpha> \beta\ge-1/2$. In fact, the right hand side of (\ref{komp-int-rep-bessel}) 
can be easily reduced to (\ref{int-rep-bessel}) by applying first polar coordinates
$u=r\cos\phi$, $v=r\sin\phi$ with $v\in[0,\sqrt{1-u^2}]$ and then the transform $z:=v/\sqrt{1-u^2}
\in[0,1]$, where the   $z$-integral fits the constants.

The following limit result  is related to
 the asymptotic expansion of the Jacobi functions in terms of Bessel functions
 in \cite{ST}; see  Lemma 3.3 in \cite{V2}.

\begin{proposition}\label{flacher-limes}
 Let $\alpha\ge \beta\ge-1/2$ with $\alpha>-1/2$ and $T>0$ a constant. There exists a constant
 $C=C(\alpha,\beta,T)$ such
that for all $\lambda\in \b R$, $t\in [0,T]$, and $n\ge 1$,
$$|\phi_{i\rho-n\lambda}^{(\alpha,\beta)}(t/n)-j_\alpha(\lambda t)|\le C\cdot |\lambda|t^2/n.$$
\end{proposition}

\begin{proof}
The integral representations in Theorem \ref{integral-representation} and (\ref{komp-int-rep-bessel})
 imply that
\begin{align}
R:=& \left|\phi_{i\rho-n\lambda}^{(\alpha,\beta)}(t/n)-j_\alpha(\lambda t)\right|
\notag\\
\le&\int_0^1\int_{0}^\pi \left|exp\left(i\lambda n\cdot \ln|ch(t/n)+re^{i\phi}\sh(t/n)|\right) 
    - e^{i\lambda tr \cos \phi}\right|\> dm_{\alpha,\beta}(r,\phi).
\notag
\end{align}
Using the  well-known inequality
$|e^{ix}-e^{iy}|\le\sqrt 2 \cdot |x-y|$ for
$x,y\in\b R$,
we obtain
$$R\le |\lambda|\cdot \int_0^1\int_{0}^\pi \left| n\cdot \ln|ch(t/n)+re^{i\phi}\sh(t/n)| -
 tr \cos \phi\right|\> dm_{\alpha,\beta}(r,\phi).
$$
As
\begin{align}
\ln|ch(t/n)+re^{i\phi}\sh(t/n)|&= 
\frac{1}{2}\ln\left(( ch(t/n)+r\cos\phi\sh(t/n))^2+r^2\sin^2\phi\sh^2(t/n\right)
\notag\\
&=\frac{1}{2}\ln\left(1+2r \cos\phi\cdot t/n+t^2(1+r^2)/n^2+
O(t^3/n^3)\right)
\notag\\
&=\frac{1}{2}\left(2r \cos\phi\frac{ t}{n} +\frac{t^2(1+r^2)}{n^2} -2r^2\cos^2\phi\frac{ t^2}{n^2} 
+O\left(\frac{t^3}{n^3}\right)\right)\notag
\end{align}
 uniformly for $t\in [0,T], r\in[0,1], \phi\in [0,\pi] $, 
 the claim follows.
\end{proof}

The next result describes the oscillatory behavior of 
$\phi_{i\rho-\lambda}^{(\alpha,\beta)}(t)$ in the spectral variable $\lambda\in\b R$ for fixed 
$\alpha,\beta$, which is uniform in $t\ge 0$.
For this we follow Section 7.2.2 of \cite{BH} and define for $k\in \b N$ the so called  moment functions
\begin{align}\label{moment-function}
m_k(t):=m_k^{(\alpha,\beta)}(t):=& \frac{\partial^k}{\partial\lambda^k}\phi_{i\rho+i\lambda}^{(\alpha,\beta)}(t)
\Bigl|_{\lambda=0}=\frac{\partial^k}{\partial\lambda^k}\phi_{-i\rho-i\lambda}^{(\alpha,\beta)}(t)
\Bigl|_{\lambda=0}
\notag\\
=&
\int_0^1\int_0^\pi \bigl(\ln\bigl(|\ch t + r\cdot e^{i\phi}\sh t|\bigr)\bigr)^k\> dm_{\alpha,\beta}(r,\phi) 
 \end{align}
for $t\ge0$ (where the second last equation follows from symmetry of the Jacobi functions in the parameter, and
last one from Theorem \ref{integral-representation}).
 In particular, the
 function $m_1$ appears as substitute of the additive function $x\mapsto x$ on the group $(\b R,+)$ on
 the Jacobi convolution structure $([0,\infty[, *_{\alpha,\beta})$; see \cite{Z1}
 and Section 7.2 of the monograph \cite{BH}.  $m_1$ can be used to define 
 a modified drift-part in a  CLT  on $[0,\infty[$; see
  \cite{Z1} \cite{BH} and Section 3 below. 
We mention that
for the parameters $\alpha,\beta$,
 for which the Jacobi functions are spherical functions of rank-one, non-compact symmetric spaces, 
this meaning of  $m_1$ is well-known for a long time
in probability theory on on hyperbolic spaces; see, for instance, \cite{KTS},  \cite{Tu}, \cite{F}.

We now derive a result which improves a general result of Zeuner
  \cite{Z1} for general Chebli-Trimeche hypergroups on $[0,\infty]$ in the special case of 
 the Jacobi convolution structures $([0,\infty[, *_{\alpha,\beta})$.
 It will be used to derive a Berry-Esseen-type CLT below.

\begin{proposition}\label{limit-momentenfkt}
Let $\alpha\ge\beta\ge-1/2$ with $\alpha>-1/2$. Then there exists a constant $C=C(\alpha,\beta)$
such that for all $t\ge 0$ and $\lambda\in\b R$,
$$|\phi_{i\rho-\lambda}^{(\alpha,\beta)}(t)- e^{i\lambda\cdot m_1(t)}|\le C(\lambda^2 +|\lambda|^3).$$
\end{proposition}

The proof depends on the following elementary observation:

\begin{lemma}\label{ln-est}
For $z\in \b C$ with $|z|\le 1$, $\epsilon\in]0,1]$, and the Euler number $e=2,71...$,
$$|\ln|1+z||\le \frac{1}{e\epsilon(1-|z|)^\epsilon}.$$
\end{lemma}

\begin{proof}
Elementary calculus yields $|x^\epsilon\cdot\ln x|\le 1/(e\epsilon)$ for $x\in]0,1]$. Therefore,
\begin{align}
|\ln|1+z||=&|\Re\ln(1+z) |\le |\ln(1+z) | = |z-z^2/2 +z^3/3\pm\ldots|
\notag\\
\le & |z|+|z|^2/2 +|z|^3/3\pm\ldots =|\ln (1-|z|)|\le \frac{1}{e\epsilon(1-|z|)^\epsilon}.
\notag
 \end{align}
\end{proof}

\begin{proof}[Proof of the Proposition:]
Let $h(t,r,\phi):=|1+re^{i\phi}\cdot \sh t/\ch t|$. Then, for $t\ge0$,
$$e^{i\lambda\cdot m_1(t)}= (\ch t)^{i\lambda}\cdot exp\left(i\lambda
\int_0^1\int_0^\pi \ln(h(t,r,\phi))\> dm_{\alpha,\beta}(r,\phi)\right) .$$
Therefore, using the integral representation of the Jacobi functions, we obtain
\begin{align}
R:=&|\phi_{i\rho-\lambda}^{(\alpha,\beta)}(t)- e^{i\lambda\cdot m_1(t)}|
\notag\\
=& \Bigl|  \int_0^1\int_0^\pi e^{i\lambda\cdot \ln(h(t,r,\phi))} \>  dm_{\alpha,\beta}(r,\phi) \> -\>
 exp\left(i\lambda
\int_0^1\int_0^\pi \ln(h(t,r,\phi))\> dm_{\alpha,\beta}(r,\phi)\right)\Bigr|.
\notag
 \end{align}
We now write down the usual power series for
 both exponentials and observe that the terms of order 0 and 1 are equal in both expansions.
Therefore,
\begin{align}
R\le&
\int_0^1\int_0^\pi\Bigl| e^{i\lambda\cdot \ln(h(t,r,\phi))} -(1+i\lambda\cdot \ln(h(t,r,\phi)) )\Bigr|
\> dm_{\alpha,\beta}(r,\phi)
\notag\\
&\quad + 
 \Bigl| 
 exp\left(i\lambda
\int_0^1\int_0^\pi \ln(h(t,r,\phi))\> dm_{\alpha,\beta}(r,\phi)\right) -1-i\lambda
\int_0^1\int_0^\pi \ln(h(t,r,\phi))\> dm_{\alpha,\beta}(r,\phi) \Bigr|.
\notag
 \end{align}
Using the well-known estimates $|\cos x-1|\le x²/2$ and  $|\sin x-x|\le |x|^3/6$ for $x\in\b R$, we obtain
$|e^{ix}-(1+ix)|\le x²/2+|x|^3/6$, and thus
$$R\le A_1^2\lambda^2/2 + A_1^3|\lambda|^3/6 + A_2\lambda^2/2 + A_3|\lambda|^3/6$$
for $A_k:=\int_0^1\int_0^\pi| \ln(h(t,r,\phi))|^k\> dm_{\alpha,\beta}(r,\phi)$, $k=1,2,3$.
In particular, by Jensen's inequality, 
\begin{equation}\label{r-abschs}
R\le A_2\lambda^2 + A_3|\lambda|^3/3.
\end{equation}
Assume now that $\alpha >\beta> -1/2$ holds. Choose some $\epsilon\in]0,1[$ with 
 $\epsilon<(\alpha -\beta)/3$ and apply Lemma \ref{ln-est} as well as
$1-r\sh t/\ch t\ge 1-r\ge (1-r²)/2$ for $r\in[0,1], t\ge0$.
 This and the definitions of $A_k$ and $h$  imply
for $k=2,3$ and some constants $C_1, C_2, C_3$ that
\begin{align}\label{eps-estimation}
 A_k &\le\int_0^1\int_0^\pi \frac{C_1}{ \cdot |1-r\sh t/\ch t|^{k\epsilon} }
 \> dm_{\alpha,\beta}(r,\phi)
 \notag\\
 &\le C_2\int_0^1\int_0^\pi   
 (1-r²)^{\alpha-\beta-1-k\epsilon}(r\sin\phi)^{2\beta}\cdot r \> dr\> d\phi
\>\> \le \>\> C_3
 \end{align}
 as claimed.

 The case $\alpha >\beta= -1/2$ follows in the same way by using $m_{\alpha,-1/2}$.

Finally,  the case
 $\alpha =\beta> -1/2$ can be reduced to the case  $\alpha >\beta= -1/2$ by the well-known quadratic transform
$\phi_{2\lambda}^{(\alpha_,\alpha)}(t)=\phi_{\lambda}^{(\alpha_,-1/2)}(2t)$ (see Eq.~(5.32) of \cite{Ko3})
 which in particular implies
 $m_1^{(\alpha_,\alpha)}(t) =\frac{1}{2}m_1^{(\alpha_,-1/2)}(2t)$. 
\end{proof}

\begin{remark}
The preceding proof leads easily to explicit constants $C=C(\alpha,\beta)$ 
in Proposition \ref{limit-momentenfkt}. For instance, for $\alpha>\beta+1$, one may take
 $\epsilon=1/3$ above and obtains from 
(\ref{eps-estimation}) and explicit values of beta-integrals that 
$A_2, A_3\le \frac{6\alpha}{e(\alpha-\beta-1)}$ and thus, by (\ref{r-abschs}), that 
$$C:=\frac{6\alpha}{e(\alpha-\beta-1)}$$
is an admissable constant in the statement of Proposition \ref{limit-momentenfkt}.
\end{remark}

We finally use the integral representation \ref{moment-function} of $m_1$ 
in order to estimate $m_1$.
Weaker estimates for $m_1$ for general Chebli-Trimeche hypergroups on
$[0,\infty[$ are given in
\cite{Z2},\cite{Z3};
 see Proposition 7.3.23 of \cite{BH}. A  better estimation is
derived in \cite{G}.

\begin{lemma}\label{est-moment}
Let $\alpha\ge\beta\ge-1/2$ with $\alpha>-1/2$. Then there exists a constant $C=C(\alpha,\beta)$
with $t-C\le m_1(t)\le t$ for all $t\ge0$.
\end{lemma}

\begin{proof}
Let $t\ge 0$, $\phi\in[0,\pi]$ and $r\in[0,1]$.
Then
$$   e^t(1-r)/2     \le |\ch t + r\cdot e^{i\phi}\sh t|\le \ch t+\sh t=e^t.$$
We conclude from
 (\ref{moment-function}) that
$$m_1(t)=\int_0^1\int_0^\pi \ln\bigl(|\ch t + r\cdot e^{i\phi}\sh t|\bigr)\> dm_{\alpha,\beta} \le
\int_0^1\int_0^\pi ln(e^t) dm_{\alpha,\beta} =t.$$
For the second inequality, we now assume $\alpha>\beta$ and conclude from Lemma \ref{ln-est}
for $\epsilon=(\alpha-\beta)/2$ that $\ln(1-r)\ge- C/(1-r)^\epsilon$ and thus
$$m_1(t)\ge \int_0^1\int_0^\pi ln( e^t(1-r)/2 )   dm_{\alpha,\beta} 
\ge t-\ln 2 -C \cdot \int_0^1\int_0^\pi(1-r)^{-\epsilon} dm_{\alpha,\beta},$$
where the latter integral has a finite value. This implies the result for $\alpha>\beta$.
The case  $\alpha=\beta>-1/2$ can again
 be handled  as in the end of the proof of
Proposition \ref{limit-momentenfkt}.
\end{proof}

 This lemma and Proposition \ref{limit-momentenfkt} lead to:

\begin{corollary}\label{limit-exp-fkt}
Let $\alpha\ge\beta\ge-1/2$ with $\alpha>-1/2$. Then there exists a constant $C=C(\alpha,\beta)$
such that for all $t\ge 0$ and $\lambda\in\b R$,
$$|\phi_{i\rho-\lambda}^{(\alpha,\beta)}(t)- e^{i\lambda\cdot t}|\le C(\lambda^2 +|\lambda|^3).$$
\end{corollary}

\section{Central limit theorems for  growing parameters}

In this section we derive two CLTs  for Jacobi random
walks, where in the first result $\alpha$ tends to infinity with fixed $\beta$, while
  in the second one $\alpha$ and $\beta$ tend to infinity.

\begin{theorem}\label{central-limit-alpha-infty}
Let $\beta\ge-1/2$ fixed, and let $(\alpha_n)_{n\ge1}\subset[\beta,\infty[$
 be an increasing sequence of parameters with 
 $\lim_{n\to\infty}n/\alpha_n=0$. Let $\nu\in M^1([0,\infty[)$ be a
 probability measure with a finite second moment $\int_0^\infty x^2 \>
d\nu(x)<\infty$ and with $\nu\ne\delta_0$, and consider the associated Jacobi random walks 
 $(S_n^{(\alpha_n,\beta)})_{n\ge0}$ on $[0,\infty[$. Then
 $$\frac{S_n^{(\alpha_n,\beta)}  -n\cdot m_1}{\sqrt n}\to N(0,m_2-m_1^2)$$
in distribution for $n\to\infty$ with a normal distribution $N(0,m_2-m_1^2)$ with parameters
$$m_1:= \int_0^\infty \ln(\ch x)\>
d\nu(x)<\infty, \quad\quad m_2:=\int_0^\infty (\ln(\ch x))^2\>
d\nu(x)\in ] m_1^2,\infty[.$$
\end{theorem}

\begin{proof}
Consider the homeomorphism
$T:[0,\infty[\to[0,\infty[, \quad t\mapsto \ln \ch t$.
We recapitulate from Proposition \ref{limit-alpha-beta-infty} that
$$\phi_{i\rho-\lambda}^{(\alpha_n,\beta)}(t)= e^{i\lambda\cdot \ln \ch t}+
O(|\lambda|/\sqrt\alpha_n)$$
uniformly in $t\in[0,\infty[$. Therefore, there exists a constant $C>0$ with
\begin{equation}\label{nu-absch1-n}
\left|\int_0^\infty \phi_{i\rho-\lambda}^{(\alpha_n,\beta)}(t)\>
d\nu^{(n;\alpha_n,\beta)}(t) -
\int_0^\infty e^{i\lambda\cdot \ln \ch t} d\nu^{(n;\alpha_n,\beta)}(t)\right|
\le
C\cdot\frac{|\lambda|}{\sqrt\alpha_n}
\end{equation}
and
\begin{equation}\label{nu-absch1-1}
\left|\int_0^\infty \phi_{i\rho-\lambda}^{(\alpha_n,\beta)}(t)\>
d\nu(t) -
\int_0^\infty e^{i\lambda\cdot \ln \ch t} d\nu(t)\right|
\le
C\cdot\frac{|\lambda|}{\sqrt\alpha_n}
\end{equation}
for $\lambda\in\b R$ and $n\in\b N$. Moreover, the random variables
 $T(S_n^{(\alpha_n,\beta)})$ have the distributions
$T(\nu^{(n;\alpha_n,\beta)})$ with the classical Fourier transforms
$$T(\nu^{(n;\alpha_n,\beta)})^\wedge(\lambda)=
\int_0^\infty e^{-i\lambda\cdot \ln \ch t}\> d\nu^{(n;\alpha_n,\beta)}(t).$$
Therefore, by (\ref{nu-absch1-n}) and (\ref{nu-absch1-1}),
\begin{align}
T(\nu^{(n;\alpha_n,\beta)})^\wedge(\lambda)&
=\int_0^\infty\phi_{i\rho+\lambda}^{(\alpha_n,\beta)}(t)
\> d\nu^{(n;\alpha_n,\beta)}(t)\> +\> O(|\lambda|/\sqrt\alpha_n)
\notag\\
&=\Bigl(\int_0^\infty\phi_{i\rho+\lambda}^{(\alpha_n,\beta)}(t)\>
d\nu(t)\Bigr)^n \> +\> O(|\lambda|/\sqrt\alpha_n)
\notag\\
&=\Bigl( T(\nu)^\wedge(\lambda)\>  +\> O(|\lambda|/\sqrt\alpha_n)\Bigr)^n
  +\> O(|\lambda|/\sqrt\alpha_n).
\notag
\end{align}
Moreover, as $\nu$ has a finite second moment by our assumption, and as 
$\ln\ch t\le t$ for $t\ge0$, the measure  $T(\nu)$ also has finite first and
second moments
$$m_k=\int_0^\infty t^k\> dT(\nu)(t)=\int_0^\infty (\ln\ch t)^k\>
d\nu(t)\quad\quad(k=1,2),$$
and thus
$$ T(\nu)^\wedge(\lambda) = 1-i\lambda m_1 -\lambda^2 m_2/2
+o(\lambda^2)\quad\quad\text{for}\quad
\lambda\to0.$$
Therefore, if we denote the distribution of $(T(S_n^{(\alpha_n,\beta)})
-n\cdot m_1)/\sqrt n$ by $\mu_n$, and if we use the assumption
$O(1/\sqrt{n\alpha_n})=o(1/n)$,
we conclude that for $\lambda\in\b R$,
\begin{align}
\mu_n^\wedge(\lambda)&= 
T(\nu^{(n;\alpha_n,\beta)})^\wedge(\lambda/\sqrt n)\cdot e^{in\cdot
  m_1\lambda/\sqrt n}
\notag\\
&=\Bigl(\Bigl( T(\nu)^\wedge(\lambda/\sqrt n)\>  +\>
O(|\lambda|/\sqrt{n\alpha_n})\Bigr)^n
+O(|\lambda|/\sqrt{n\alpha_n})\Bigr)
\cdot e^{in\cdot
  m_1\lambda/\sqrt n}
\notag\\
&=\Bigl(1-\frac{i\lambda m_1}{\sqrt n} -\frac{\lambda^2 m_2}{2n}+o(1/n) +
O(|\lambda|/\sqrt{n\alpha_n})\Bigr)^n\cdot\Bigl(1+\frac{i\lambda m_1}{\sqrt n}
- \frac{\lambda^2 m_1^2}{2n} +o(1/ n)\Bigr)^n
\notag\\
&=\Bigl(1-\frac{\lambda^2 (m_2-m_1^2)}{2n} +o(1/ n)\Bigr)^n,
\notag
\end{align}
which tends for $n\to\infty$ to
$e^{-\lambda^2(m_2-m_1^2)/2}=N(0,m_2-m_1^2)^\wedge(\lambda)$.
The classical continuity theorem of Levy yields that 
$(T(S_n^{(\alpha_n,\beta)})-nm_1)/\sqrt n$ tends in distribution to
$N(0,m_2-m_1^2)$.
This in particular shows that $\ln(\ch(S_n^{(\alpha_n,\beta)}))/n\to m_1>0$ and thus
 $e^{-2S_n^{(\alpha_n,\beta)}}\to 0$    in probability.
Using
$$x-\ln 2\le \ln(\ch x)\le x+\ln(1+e^{-2x})\le x+e^{-2x}$$
and thus
$$\ln\ch S_n^{(\alpha_n,\beta)} - e^{-2S_n^{(\alpha_n,\beta)}} \le  S_n^{(\alpha_n,\beta)} \le \ln\ch S_n^{(\alpha_n,\beta)}+\ln 2,$$
we obtain that $(S_n^{(\alpha_n,\beta)}-nm_1)/\sqrt n$ tends in distribution to
$N(0,m_2-m_1^2)$ as claimed.
\end{proof}

\begin{remark}
The preceding theorem was derived  in \cite{V3} by completely
  different methods under the stronger condition $n/\sqrt \alpha_n\to0$ 
for $n\to\infty$. Recently, the preceding theorem was generalized by 
 W.~Grundmann \cite{G} to an arbitrary sequence 
$(\alpha_n)_n$ with $\alpha_n\to\infty$.
\end{remark}

The following CLT can be proved in the same way as Theorem
\ref{central-limit-alpha-infty} by using Proposition
 \ref{limit-alpha-beta-infty}
and the  homeomorphism $T:[0,\infty[\to[0,\infty[$ with $T(x):=\ln\sqrt{\ch^2 x- (1/c)\sh^2x} $
instead of  Proposition \ref{limit-alpha-infty} and $T(x):=\ln \ch x$.
We expect that it can also be generalized to an arbitrary sequence 
$(\alpha_n)_n$ with $\alpha_n\to\infty$ similar to \cite{G}.

\begin{theorem}\label{central-limit-alpha-beta-infty}
Fix constants $c>1$ and $d>0$, and let $(\beta_n)_{n\ge1}\subset[\beta,\infty[$
 be an increasing sequence of parameters with 
 $\lim_{n\to\infty}n/\beta_n=0$. Moreover, put $\alpha_n:=c\beta_n+d$.

Let $\nu\in M^1([0,\infty[)$ be a
 probability measure with a finite second moment $\int_0^\infty x^2 \>
d\nu(x)<\infty$ and with $\nu\ne\delta_0$, and consider the associated Jacobi random walks 
 $(S_n^{(\alpha_n,\beta_n)})_{n\ge0}$ on $[0,\infty[$.
Then
 $$\frac{S_n^{(\alpha_n,\beta_n)}  -n\cdot m_1}{\sqrt n}\to N(0,m_2-m_1^2)$$
in distribution for $n\to\infty$ with a normal distribution $N(0,m_2-m_1^2)$ with parameters
$$ m_1:=\int_0^\infty \ln\sqrt{\ch^2 x- (1/c)\sh^2x}\>
d\nu(x)>0,$$
$$ m_2:=\int_0^\infty (\ln\sqrt{\ch^2 x- (1/c)\sh^2x})^2\>
d\nu(x)\in ] m_1^2,\infty[.$$
\end{theorem}

\section{Central limit theorems for fixed parameters}

In this section we present a couple of CLTs for fixed parameters
$\alpha,\beta$.
We consider the following setting: We fix some  non-trivial
probability measure $\nu\in M^1([0,\infty[ )$   with $\nu\ne\delta_0$ which
    possibly satisfies some moment condition.
For each $d\in]0,1]$ consider the
compressing map $D_{d}:x\mapsto d x$ on $[0,\infty[$ as well as the
    compressed measure $\nu_{ d}:= D_{d}(\nu)   \in M^1([0,\infty[)$. For 
given $\nu$ and $d$ we consider a Jacobi random walk $(S_n^{(\alpha,\beta,
  d)})_{n\ge0}$ on $[0,\infty[$ associated with the law $\nu_{ d}$ as
  above.
We now investigate the limit behavior of $(S_n^{(\alpha,\beta,  n^{-r})})_{n\ge0}$
for different powers $r\ge0$. 
The most classical cases appear for $r=1/2$ and $r=0$.

In fact,  for $r=1/2$, the random variables $S_n^{(\alpha,\beta,  n^{-1/2})}$ tend in distribution
to some probability measure $\gamma_{t_0}^{(\alpha,\beta)}$ which is part of the
unique (up to time parametrization) Gaussian convolution semigroup 
 $(\gamma_t^{(\alpha,\beta)})_{t\ge0}$ on the  Jacobi hypergroup on
$[0,\infty[$ where the correct  $t_0$ mainly depends on the second moment
$m_2:=\int_0^\infty x^2\> d\nu(x)$ of $\nu$. For details see Theorem 7.4.1 of \cite{BH}.
We remark that this CLT holds for general Chebli-Trimeche hypergroups
on $[0,\infty[$.
As the proof is very standard and universal, we do not treat this
case here.

The case  $r=0$ was handled by Zeuner \cite{Z1}  for
  Chebli-Trimeche hypergroups on $[0,\infty[$ with exponential growth
     by using an estimation  weaker than Proposition
     \ref{limit-momentenfkt}. We here reprove this CLT together with a
     Berry-Esseen-type
order of convergence $O(n^{-1/3})$ which is slightly worse than the  order  $O(n^{-1/2})$ in 
the
classical CLT for sums of iid random variables on $\b R$. 

A CLT in the case $r>1/2$ was treated earlier in 
 in Section 3 of \cite{V2} on the basis of a variant of 
Proposition  \ref{flacher-limes}.
We here  derive this CLT on the basis of Proposition
\ref{flacher-limes}. The reason for doing so is, that in our opinion, the
proof of Proposition  \ref{flacher-limes} above is more elementary than
the corresponding variant in Section 3 of  \cite{V2}, and that the present
 proof can be transfered without problems
to a higher dimensional setting.

We finally turn to the case $r\in]0,1/2[$ which was not handled before.
It  turns out that the cases $r\in]0,1/6[$, $r=1/6$, and $r\in]1/6,1/2[$
lead to a  different limit behavior.

Before going into details, we collect some properties of the moment functions
$m_k$ introduced in (\ref{moment-function}) from the literature.

\begin{lemma}\label{properties-moment-function}
\begin{enumerate}
\item[\rm{(1)}] $m_1$ is increasing on $[0,\infty[$.
\item[\rm{(2)}] There is a constant $R>0$ with $m_1(x)\le Rx^2$ for all $x\ge 0$.
\item[\rm{(3)}] For all $\mu,\nu\in M^1([0,\infty[)$ with finite first moment,
$$\int_0^\infty m_1 \> d(\mu*\nu)= 
\int_0^\infty m_1 \> d\mu +\int_0^\infty m_1 \> d\nu.$$
\item[\rm{(4)}] For $k\in\b N$ and $t\ge0$, $m_k(x)\le x^k$.
\end{enumerate}
\end{lemma}

\begin{proof} For (1), (3), and (4) see \cite{Z2} or Section 7.2 of \cite{BH}. Part (2)
  follows from the fact that $m_1$ is analytic with $m_1^\prime(0)=0$ in
  combination with Lemma \ref{est-moment}.
\end{proof}

\begin{theorem}\label{central-limit-momentenfkt}
Let $\nu\in M^1([0,\infty[ )$  with $\nu\ne
    \delta_0$ and
 finite second  moment and $(S_n^{(\alpha,\beta,1)})_n$ the associated Jacobi random walk
without initial compression.  For $k\in\b N$ let $$M_k:=\int_0^\infty m_k\> d\nu(t)\in]0,\infty]$$
be the so called modified moments of $\nu$ with $M_1,M_2<\infty$ by our assumption.
Then  $$(S_n^{(\alpha,\beta,1)}-nM_1)/\sqrt n$$
tends in distribution to $N(0, M_2-M_1^2)$.

Moreover, if  $\nu\in M^1([0,\infty[ )$ in addition has third moments, then
 the distribution functions of $(S_n^{(\alpha,\beta,1)}-nM_1)/\sqrt n$
tend uniformly to the distribution function of
$N(0, M_2-M_1^2)$ of order $n^{-1/3}$.
\end{theorem}

\begin{proof}
For the proof of the first statement on the basis of 
Corollary \ref{limit-exp-fkt} we refer to Theorem 7.4.2 of \cite{BH}.

For the second statement fix a small constant
$c>0$  and put $T:=cn^{1/3}$.
 Moreover, let $n\in\b N$ and $\lambda\in\b R$ with $|\lambda|\le T$.
As $\nu$ has a finite third moment, the function 
$\tilde\nu:\b R\to\b C, \> \lambda\mapsto \int_0^\infty 
\phi^{(\alpha,\beta)}_{i\rho-\lambda}(t)\> d\nu(t)$ is three times differentiable
 (see  \cite{Z1} or 7.2.19 of \cite{BH}), i.e., by the remainder in the Taylor expansion,
$$\left|\tilde\nu(\lambda/\sqrt n)- \left( 1+i\lambda M_1/\sqrt n - \lambda^2 M_2/(2n)\right)\right|
\le \frac{\lambda^3M_3}{6n^{3/2}}$$
for $n\in\b N$ and $\lambda\in\b R$. Therefore, for suitable positive  constants $C_1,C_2,\ldots$,
$$\left|e^{-i\lambda M_1/\sqrt n }\tilde\nu(\lambda/\sqrt n)- \left( 1-\frac{\lambda^2}{2n}(M_2-M_1^2)\right)\right|
\le C_1\frac{\lambda^3}{n^{3/2}},$$
and in particular, for $|\lambda|\le T$ and $c>0$ sufficiently small,
$$|\tilde\nu(\lambda/\sqrt n)|\le 1-\frac{\lambda^2}{2n}(M_2-M_1^2)+  C_1\frac{|\lambda|^3}{n^{3/2}}
\le 1-\frac{C_2\lambda^2}{n}\le e^{-C_2\lambda^2/n}.$$
Moreover, under this restriction and by the same arguments,
$$\left|e^{-\lambda^2(M_2-M_1^2)/(2n)}-\left( 1-\frac{\lambda^2}{2n}(M_2-M_1^2)\right)\right|\le 
 C_3\frac{\lambda^4}{n^{2}}$$
and $|e^{-\lambda^2(M_2-M_1^2)/(2n)}|\le e^{-C_2\lambda^2/n}$.
Therefore, using $|a^n-b^n|\le n|a-b|\cdot\max(|a|,|b|)^{n-1}$, we get
\begin{align}\label{Berry}
&\left|e^{-i\lambda M_1\sqrt n }\tilde\nu(\lambda/\sqrt n)^n - e^{-\lambda^2(M_2-M_1^2)/2}\right|\notag\\
&\quad\quad\le n\cdot\left| e^{-i\lambda M_1/\sqrt n }\tilde\nu(\lambda/\sqrt n) -e^{-\lambda^2(M_2-M_1^2)/(2n)}\right|\cdot
C_4\cdot e^{-C_5\lambda^2}
\notag\\&\quad\quad\le 
C_6\cdot\frac{\lambda^3}{n^{1/2}}\cdot e^{-C_5\lambda^2}
.
\end{align}
On the other hand, by Corollary \ref{limit-exp-fkt} and the multiplicativity of the Jacobi functions,
$$\tilde\nu(\lambda/\sqrt n)^n=\int_0^\infty 
\phi^{(\alpha,\beta)}_{i\rho-\lambda}(t)\> d\nu^(n)(t) =\int_0^\infty e^{i\lambda t}\> d\nu^{(n)}(t)
+O(\lambda^2+|\lambda|^3).$$
Therefore, by (\ref{Berry}), the usual Fourier transform $\widehat{\nu^{(n)}}$ of $\nu^{(n)}$ satisfies
\begin{align}\label{Berry1}
\int_{-T}^T& \frac{|\widehat{\nu^{(n)}}(\lambda/\sqrt n)\cdot e^{i\lambda M_1\cdot \sqrt n }
 - e^{-\lambda^2(M_2-M_1^2)/2}|}{|\lambda|}\> d\lambda 
\notag\\
&\quad\quad \le C_7\int_{-T}^T \left(\frac{|\lambda|}{n} +\frac{|\lambda^2|}{n^{3/2}} +
\frac{\lambda^2+|\lambda|^3}{n^{1/2}}\cdot e^{-C_5\lambda^2}\right) \> d\lambda 
\end{align}
As this expression is of order $O(n^{-1/3})$ for $T=cn^{1-/3}$, we conclude
from the lemma of Berry-Esseen  (see, for instance, Lemma 2 in Section
 XVI.3 of Feller \cite{Fe}) that the distribution functions of $(S_n^{(\alpha,\beta,1)}-nM_1)/\sqrt n$
tend uniformly to the distribution function of
$N(0, M_2-M_1^2)$ of order $n^{-1/3}$ as claimed.
\end{proof} 

\begin{theorem}\label{central-flacher-limes}
Let $\nu\in M^1([0,\infty[ )$  with $\nu\ne
    \delta_0$ and
 finite second moment
$m_2:= \int_0^\infty x^2\> d\mu(x)\in ]0,\infty[$.
Let $\alpha\ge\beta\ge-1/2$ with $\alpha>-1/2$, and let $r>1/2$.
Then
  $$ \left(\frac{2(\alpha+1)}{m_2}\right)^{1/2} n^{r-1/2} \cdot S_{n}^{(\alpha,\beta, n^{-r})}$$ 
 tends
 in distribution  to the Rayleigh distribution $\rho_\alpha$ with Lebesgue
 density
$$\frac{1}{2^\alpha \Gamma(\alpha+1)} x^{2\alpha+1} e^{-x^2/2} \quad\quad(x\ge 0).$$
\end{theorem}

The proof needs some preparations. The following estimation is proved in Lemma
3.5 of \cite{V2} by using Lemma \ref{properties-moment-function}

\begin{lemma}
 There exists a constant $M=M(\alpha,\beta,r,\nu)>0$ such that 
$${\bf P}(S_{n}^{(\alpha,\beta, n^{-r})}\ge c)\>\le\> \frac{M}{m_1(c) n^{2r-1}}
\quad\text{ for}\quad c>0,\> n\in\b N.$$
\end{lemma}

 For the rest of the proof of  Theorem \ref{central-flacher-limes} we denote the distribution of
$n^{r-1/2}S_{n}^{(\alpha,\beta, n^{-r})}$ by $\mu_n$. The proofs of the
 following two lemmas are similar to, but slightly different from those of Lemmas 
3.6 and 3.7 of \cite{V2}.

\begin{lemma}\label{hilflemma-est1} For all $\lambda\ge0$,
$$\lim_{n\to\infty} \int_0^\infty\Bigl|j_\alpha(\lambda t) -
\phi_{\lambda n^{r-1/2}}^{(\alpha,\beta)}(tn^{1/2-r}) \Bigr|
d\mu_n(t) \>=\> 0.$$
\end{lemma}

\begin{proof}
Let $c>0$. Then, by Proposition \ref{flacher-limes},
the boundedness of the involved Jacobi and Bessel functions and by the
preceding lemma,
\begin{align}\label{balanced-ineq}
 A_n:= \int_0^\infty&\Bigl|j_\alpha(\lambda t) -\phi_{\lambda n^{r-1/2}}^{(\alpha,\beta)}(tn^{1/2-r}) \Bigr|
d\mu_n(t) =\Biggl(\int_0^c +\int_c^\infty\Biggr)\Bigl| \ldots\Bigr|d\mu_n(t)
\notag\\
&\le \frac{H(c)|\lambda|}{ n^{r-1/2}}+ 2\cdot
{\bf P}(n^{r-1/2}S_{n}^{(\alpha,\beta, n^{-r})}\ge c)
\notag\\
&\le\frac{H(c)|\lambda|}{ n^{r-1/2}}+
\frac{2M}{m_1(cn^{1/2-r})\cdot n^{2r-1}} 
\end{align}
for some constants $M$ and some  $H(c)$ depending on $c$.
On the other hand, as $m_1^{\prime\prime}(0)>0$, there exist $a,b>0$ with
$m_1(x)\ge ax^2$ for $x\in[0,b]$. 

Now let $\epsilon>0$. Then choose $c$ with $2M/(ac^2)\le\epsilon/2$, and now
$n$ large enough with
$$\frac{H(c)|\lambda|}{ n^{r-1/2}}\le\epsilon/2 \quad\quad\text{and}\quad\quad
c n^{1/2-r}\le b.$$
As the latter implies 
$$\frac{2M}{m_1(cn^{1/2-r})\cdot n^{2r-1}} \le \frac{2M}{ac^2}\le\epsilon/2,$$
we obtain from (\ref{balanced-ineq}) that $A_n\le\epsilon$ for large $n$ as claimed.
\end{proof}

\begin{lemma}\label{hilflemma-est2}
 Let   $\lambda>0$ and  $\nu\in M^1([0,\infty[)$ with finite
      second moment $m_2<\infty$.
 Then, for $n\to\infty$,
$$\int_0^\infty \phi_{ \lambda n^{r-1/2}}^{(\alpha,\beta)}(t/n^{r})\>d\nu(t) \>=\>
1-\frac{\lambda^2m_2}{ 4(\alpha+1)n} + o(1/n).$$
\end{lemma}

\begin{proof} Consider 
$H_\lambda(t):=|\phi_\lambda^{(\alpha,\beta)}(t)-j_\alpha(\lambda t)|\le2$,
We apply Proposition \ref{flacher-limes} to $t/n^r\le 1$ instead of $t$ with
$n=1$ there. Therefore, for  some $C>1$, and by Markov's inequality, 
\begin{align}\label{est-hilfi}
 \int_0^\infty H_{ \lambda n^{r-1/2}}(t/n^r)\> d\nu(t) &\le \int_0^{n^r}H_{
   \lambda n^{r-1/2}}(t/k^r) \>  d\nu(t)  \> +\>  2\nu([n^r,\infty[)
\notag\\
&\le C\lambda\cdot\int_0^{n^r}\frac{t^2 n^{r-1/2}}{n^{2r}}\>  d\nu(t)  \> +\>  2\nu([n^r,\infty[)
\notag\\
&\le C\lambda m_2/n^{r+1/2} +2m_2/n^{2r} \quad=\quad o(1/n).
\end{align}
Furthermore, as $m_2<\infty$, the Hankel transform
 $g(\lambda):=\int_0^\infty j_\alpha(\lambda t)\> d\nu(t)$ of $\nu$ is
two-times differentiable at $\lambda=0$ with
 $$g(\lambda) = 1-\frac{\lambda^2m_2}{4(\alpha+1)} + o(\lambda^2)\quad{\rm for} \quad \lambda\to 0$$
(see Theorem 4.7 of \cite{Z1} or Section 7.2 of \cite{BH})
 with $m_1=0$ and $m_2(x)=x^2/(\alpha+1)$ there). Hence, by (\ref{est-hilfi}),
\begin{align}
\int_0^\infty \phi_{ \lambda n^{r-1/2}}^{(\alpha,\beta)}(t/n^{r})\> d\mu(t)&=
\int_0^\infty j_\alpha(\lambda t/n^{1/2})\>d\mu(t) +o(1/n)
\\
&= 1-\frac{x^2\lambda^2}{4 (\alpha+1) n}  + o(1/n)
\notag\end{align}
as claimed.
\end{proof}

\begin{proof}[Proof of Theorem \ref{central-flacher-limes}]
 Fix $\lambda\in [0,\infty[$. Then Lemmas \ref{hilflemma-est1} and
\ref{hilflemma-est2} lead to
\begin{align}
\lim_{n\to\infty}& \int_{0}^\infty j_\alpha(\lambda t)\> d\mu_n(t)=
\lim_{n\to\infty} \int_{0}^\infty \phi_{ xn^{r-1/2}}^{(\alpha,\beta)}(tn^{1/2
  -r})\> d\mu_n(t)
\notag\\
&=
\lim_{n\to\infty} \int_{0}^\infty \phi_{ \lambda
  n^{r-1/2}}^{(\alpha,\beta)}(t)\>
 d\nu_{n^{-r}}^{(n)}(t) =
\lim_{n\to\infty}\Biggl(\int_{0}^\infty \phi_{ \lambda
  n^{r-1/2}}^{(\alpha,\beta)}(t)\>
 d\nu_{n^{-r}}(t)\Biggr)^n
\notag\\
&=
\lim_{n\to\infty}\Biggl(\int_{0}^\infty \phi_{ \lambda
  n^{r-1/2}}^{(\alpha,\beta)}(t/n^r)\>
 d\nu(t)\Biggr)^n
=
\lim_{n\to\infty}\Biggl(1-\frac{x^2m_2}{ 4(\alpha+1)n} +
 o(1/n)\Biggr)^n 
 \notag\\
&= e^{-\lambda^2m_2/4(\alpha+1)}.
\end{align}
As  the  Rayleigh distribution $\rho_\alpha$ satisfies 
$$\int_{0}^\infty j_\alpha(\lambda t)\> d\rho_\alpha(t)=e^{-\lambda^2/2},$$
Levy's continuity theorem for the Hankel transform
 (see, for instance, Section 4.2 of \cite{BH}) now yields that
$ \left(\frac{2(\alpha+1)}{m_2}\right)^{1/2} n^{r-1/2} \cdot S_{n}^{(\alpha,\beta, n^{-r})}$
 tends to $\rho_\alpha$ as claimed.
\end{proof}

\begin{theorem}\label{central-neuer-fall}
Let $\nu\in M^1([0,\infty[ )$  with $\nu\ne
    \delta_0$,with  compact support, and thus with finite
 moments
$m_k:= \int_0^\infty x^k\> d\mu(x)\in ]0,\infty[$ ($k\in\b N$).
Let $\alpha\ge\beta\ge-1/2$ with $\alpha>-1/2$ and $r\in ]0,1/2[$. Then the random variables
$S_{n}^{(\alpha,\beta, n^{-r})}$ have the following behavior for $n\to\infty$.
\begin{enumerate}
\item[\rm{(1)}] If $r\in]1/6, 1/2[$, then
 $$\frac{ S_{n}^{(\alpha,\beta, n^{-r})} -\frac{\rho m_2}{2(\alpha+1)}\cdot n^{1-2r}}{ n^{1/2-r}}$$
tends in distribution to $N(0, \frac{m_2}{2(\alpha+1)})$.
\item[\rm{(2)}]  If $r=1/6$, then 
$$\frac{ S_{n}^{(\alpha,\beta, n^{-1/6})} -\frac{\rho m_2}{2(\alpha+1)}\cdot n^{2/3}}{ n^{1/3}}$$
tends in distribution to
 $N(
-\frac{\rho(\alpha+3\beta+2)m_4}{12(\alpha+1)(\alpha+2)}
, \frac{m_2}{2(\alpha+1)})$.
\item[\rm{(3)}] If $r\in ]0,1/6[$, then
$$\frac{ S_{n}^{(\alpha,\beta, n^{-r})} -\frac{\rho m_2}{2(\alpha+1)}\cdot n^{1-2r}}{ n^{1-4r}} \longrightarrow
  -\frac{\rho(\alpha+3\beta+2)m_4}{12(\alpha+1)(\alpha+2)}    $$
in probability.
\end{enumerate}
\end{theorem}

Notice that the drift term $ -\frac{\rho(\alpha+3\beta+2)m_4}{12(\alpha+1)(\alpha+2)} $ in (2) and (3) is 
negative, and that the case (2) combines the features of the cases (1) and (3). 

The proof will be based on a simple Taylor-type expansion of of the $\phi_\lambda^{(\alpha,\beta)}$ which 
is an immediate consequence of the well-known representation of 
 $\phi_\lambda^{(\alpha,\beta)}$ as hypergeometric series.
 
\begin{lemma}
Let $\alpha\ge\beta\ge-1/2$ with $\alpha>-1/2$, and let $\lambda\in \b R$ and $t\ge0$.
Then, for $a,r>0$ and $n\to\infty$,
\begin{align}
\phi_{i\rho-\lambda/n^a}^{(\alpha,\beta)}(t/n^r)&=
1+ \frac{i\rho\lambda t^2}{2(\alpha+1)n^{a+2r}} -\frac{\lambda^2 t^2}{4(\alpha+1)n^{2a+2r}}
\notag\\
&\quad   -\frac{i\rho(\alpha+3\beta+2)t^4\lambda}{12(\alpha+1)(\alpha+2)n^{a+4r}}
 +O(n^{-a-6r}) +O(n^{-2a-4r})
\notag
\end{align}
locally uniformly in $t\in[0,\infty[$.
\end{lemma}

\begin{proof} Using $\rho=\alpha+\beta+1$ and Eq.~(2.4) of \cite{Ko3}, we have
\begin{align}
\phi_{i\rho-\lambda/n^a}^{(\alpha,\beta)}(t/n^r)&=
\quad_2F_1\left(\rho+i\lambda/(2n^a), -i\lambda/(2n^a); \alpha+1; -\sh^2(t/n^r)\right)
\notag\\
&=\quad 1+  \frac{(\rho+i\lambda/(2n^a))\cdot i\lambda}{2n^a\cdot (\alpha+1)}\cdot \sh^2(t/n^r)
\notag\\
&\quad\quad  - \frac{(\rho+i\lambda/(2n^a))(\rho+1+i\lambda/(2n^a)) i\lambda(1-i\lambda/(2n^a))}{
4n^a\cdot (\alpha+1)(\alpha+2)}\cdot \sh^4(t/n^r)
\notag\\
&\quad\quad 
 + O(n^{-a-6r})
\notag\\
&=\quad 1+  \frac{(\rho+i\lambda/(2n^a))\cdot i\lambda}{2n^a\cdot (\alpha+1)}\cdot \left(\frac{t}{n^r}+
\frac{t^3}{6n^{3r}}\right)^2
\notag\\
&\quad\quad  - \frac{(\rho+i\lambda/(2n^a))(\rho+1+i\lambda/(2n^a)) i\lambda(1-i\lambda/(2n^a))}{
4n^a\cdot (\alpha+1)(\alpha+2)}\cdot \frac{t^4}{n^{4r}}
\notag\\
&\quad\quad 
+O(n^{-a-6r}) +O(n^{-2a-4r})
\notag
\end{align}
locally uniformly in $t$ which readily leads to the claim.
\end{proof}

This expansion leads immediately to:

\begin{corollary}
Let $\nu\in M^1([0,\infty[ )$  with compact support. Then, in the setting of the preceding lemma,
\begin{align}
\int_0^\infty \phi_{i\rho-\lambda/n^a}^{(\alpha,\beta)}(t/n^r) \> d\nu(t)=
& 1+ \frac{i\rho\lambda m_2}{2(\alpha+1)n^{a+2r}} -\frac{\lambda^2 m_2}{4(\alpha+1)n^{2a+2r}}
\notag\\
&\quad -\frac{i\rho(\alpha+3\beta+2)m_4\lambda}{12(\alpha+1)(\alpha+2)n^{a+4r}}
  +O(n^{-a-6r}) +O(n^{-2a-4r}).
\notag
\end{align}
\end{corollary}

\begin{proof}[Proof of Theorem \ref{central-neuer-fall}]
We first recapitulate for all cases that by Corollary \ref{limit-exp-fkt},
\begin{equation}\label{estim-hilf}
|\phi_{i\rho-\lambda}^{(\alpha,\beta)}(t)- e^{i\lambda\cdot t}|=O(\lambda^2 +|\lambda|^3)
\end{equation}
uniformly in $t$. We now consider the different cases:

\begin{enumerate}
\item[\rm{(1)}] Let $r\in]1/6, 1/2[$. In this case we put $a:=1/2-r\in]0,1/3[$ and observe that,
due to $r\ge1/6$, $1=2a+2r<a+4r$. Therefore, Eq.~(\ref{estim-hilf}), the definition of the 
measures $\nu_{n^{-r}}$ above, the multiplicativity of Jacobi functions, and the preceding corollary
imply that for all $\lambda\in\b R$ and $n\in \b N$,
 the classical  Fourier transform of the distribution $\mu_n:=\nu_{n^{-r}}^{(n)}$ of
the random variable $ S_{n}^{(\alpha,\beta, n^{-r})}$ satisfies
\begin{align}
\hat\mu_n(\lambda/n^a)&=\int_0^\infty e^{-it\lambda/n^a}\> d\nu_{n^{-r}}^{(n)}(t)\notag \\
&=
\int_0^\infty \phi_{i\rho+\lambda/n^a}^{(\alpha,\beta)}(t) \> d\nu_{n^{-r}}^{(n)}(t) +o(1)
\notag \\
&=\left(\int_0^\infty \phi_{i\rho+\lambda/n^a}^{(\alpha,\beta)}(t) \> d\nu_{n^{-r}}(t)\right)^n
 +o(1)
\notag \\
&=\left(\int_0^\infty \phi_{i\rho+\lambda/n^a}^{(\alpha,\beta)}(t/n^r) \> d\nu(t)\right)^n
 +o(1)
\notag \\
&=\left(1 - \frac{i\rho\lambda m_2}{2(\alpha+1)n^{a+2r}} -\frac{\lambda^2 m_2}{4(\alpha+1)n}
+o(1/n)\right)^n
 +o(1).
\notag
\end{align}
Therefore, using $a=1/2-r$,
\begin{align}
\hat\mu_n&(\lambda/n^a)\cdot exp\left(\frac{i\rho m_2\lambda n^{1/2-r}}{2(\alpha+1)}\right)
\notag \\
&= \left(1 - \frac{i\rho\lambda m_2}{2(\alpha+1)n^{1/2+r}} -\frac{\lambda^2 m_2}{4(\alpha+1)n}
+o(1/n)\right)^n\cdot
\notag \\
&\quad\quad \cdot
\left(1+\frac{i\rho\lambda m_2}{2(\alpha+1)n^{1/2+r}}+ O(n^{-1-2r})
\right)^n + o(1)
\notag \\
&= \left(1-\frac{\lambda^2 m_2}{4(\alpha+1)n} +o(1/n)\right)^n + o(1)
\notag
\end{align}
which tends for $\lambda\in\b R$ to $exp(-\lambda^2 m_2/(4(\alpha+1))$. 
Levy's continuity theorem for the classical Fourier transform now implies that
 $$\frac{ S_{n}^{(\alpha,\beta, n^{-r})} -\frac{\rho m_2}{2(\alpha+1)}\cdot n^{1-2r}}{ n^{1/2-r}}$$
tends in distribution to $N(0, \frac{m_2}{2(\alpha+1)})$ as claimed.
\item[\rm{(2)}] Now let $r=1/6$. Here we put $a:=1/6$. Then $2a+2r=a+4r=1$ and $a+2r=2/3$.
As in the first case, we obtain from the preceding corollary
\begin{align}
\hat\mu_n&(\lambda/n^{1/3})\cdot exp\left(\frac{i\rho m_2\lambda n^{1/3}}{2(\alpha+1)}\right)
\notag \\
&= \left(1 - \frac{i\rho\lambda m_2}{2(\alpha+1)n^{2/3}} -\frac{\lambda^2 m_2}{4(\alpha+1)n}
+ \frac{i\rho(\alpha+3\beta+2)\lambda m_4}{12(\alpha+1)(\alpha+2)n} +
o(1/n)\right)^n\cdot
\notag \\
&\quad\quad \cdot
\left(1+\frac{i\rho\lambda m_2}{2(\alpha+1)n^{2/3}}+ O(n^{-4/3})
\right)^n + o(1)
\notag
\end{align}
which tends for $\lambda\in\b R$ to $exp\left(-\frac{\lambda^2 m_2}{4(\alpha+1)}+
\frac{i\rho(\alpha+3\beta+2)\lambda m_4}{12(\alpha+1)(\alpha+2)}\right)$. 
The proof is now completed as before.
\item[\rm{(3)}] Let $r\in ]0,1/6[$.  We here put $a:=1-4r\in]1/3,1[$, and obtain
$2a+2r>a+4r=1$ and $a+2r=1-2r$. As in the first and second part, we obtain
\begin{align}
\hat\mu_n&(\lambda/n^{a})\cdot exp\left(\frac{i\rho m_2\lambda n^{2r}}{2(\alpha+1)}\right)
\notag \\
&= \left(1 - \frac{i\rho\lambda m_2}{2(\alpha+1)n^{1-2r}} 
+ \frac{i\rho(\alpha+3\beta+2)\lambda m_4}{12(\alpha+1)(\alpha+2)n} +
o(1/n)\right)^n\cdot
\notag \\
&\quad\quad \cdot\left(1+\frac{i\rho\lambda m_2}{2(\alpha+1)n^{1-2r}}+ O(n^{-2+4r})
\right)^n + o(1)
\notag
\end{align}
which tends for $\lambda\in\b R$ to $exp\left(
\frac{i\rho(\alpha+3\beta+2)\lambda m_4}{12(\alpha+1)(\alpha+2)}\right)$.
Therefore, again by Levy's continuity theorem,
$$\frac{ S_{n}^{(\alpha,\beta, n^{-r})} -\frac{\rho m_2}{2(\alpha+1)}\cdot n^{1-2r}}{ n^{1-4r}} \longrightarrow
- \frac{\rho(\alpha+3\beta+2)m_4}{12(\alpha+1)(\alpha+2)}$$
in distribution and thus in probability.
\end{enumerate}
\end{proof}

\end{document}